\documentclass[12pt]{amsart}
\usepackage{mathrsfs}
\usepackage{amsmath,amssymb,amsthm,upref,graphicx,mathrsfs}
\usepackage{enumerate}

\usepackage{color}
\usepackage[
  colorlinks=true,
  linkcolor=blue,
  citecolor=blue,
  urlcolor=blue]{hyperref}

\newtheorem{theorem}{Theorem}

\newtheorem{lemma}[theorem]{Lemma}
\newtheorem{corollary}[theorem]{Corollary}

\theoremstyle{definition}

\newtheorem{remark}[theorem]{Remark}

\newcommand{\be}{\begin{eqnarray*}}
\newcommand{\ee}{\end{eqnarray*}}
\newcommand{\beq}{\begin{equation}}
\newcommand{\eeq}{\end{equation}}

\begin{document}

\title[Noncommutative sharp  dual Doob inequalities]
{Noncommutative sharp  dual Doob inequalities}

\authors


\author[Sukochev]{Fedor Sukochev}
\address{School of Mathematics and Statistics,
UNSW, Austrilia}
\email{f.sukochev@unsw.edu.au}

\author[Zhou]{Dejian Zhou}
\address{School of Mathematics and Statistics,
Central South University, Changsha 410083, China}
\email{zhoudejian@csu.edu.cn}
%


%
%
\begin{abstract}
Let $(x_k)_{k=1}^n$ be positive elements in the noncommutative Lebesgue space $L_p(\mathcal{M})$, and let $(\mathcal{E}_k)_{k=1}^n$ be a sequence of conditional expectations with respect to a increasing subalgebras $(\mathcal{M}_n)_{k\geq1}$ of the semifinite von Neumann algebra $\mathcal{M}$. We establish the following sharp noncommutative  dual Doob inequalities:
\begin{equation*}
	\Big\| \sum_{k=1}^nx_k\Big\|_{L_p(\mathcal{M})}\leq \frac{1}{p} \Big\| \sum_{k=1}^n\mathcal{E}_k(x_k)\Big\|_{L_p(\mathcal{M})},\quad 0<p\leq 1,
\end{equation*}
and 
\begin{equation*}
	\Big\| \sum_{k=1}^n\mathcal{E}_k(x_k)\Big\|_{L_p(\mathcal{M})}\leq p\Big\| \sum_{k=1}^nx_k\Big\|_{L_p(\mathcal{M})},\quad 1\leq p\leq 2.
\end{equation*}
As applications, we also obtain several sharp noncommutative square function  and Doob inequalities.
\end{abstract}

\maketitle

%
%
\section{Introduction and Main results}

The study of noncommutative martingale theory began in 1970s (see e.g. \cite{Cu1971}).  Since  Pisier and Xu published their remarkable paper \cite{PX1997}, where the noncommutative Burkholder-Gundy inequality was established, lots of classical  martingale inequalities have been extended to the noncommutative setting. We refer the reader to   \cite{Ju2002} for the noncommutative Doob maximal inequality,  to \cite{Rand2002} for the weak type $(1,1)$ inequality of martingale transform;
to  \cite{PR2006} for the noncommutative Gundy decomposition. 
We also refer the reader to  \cite{JX2003}, \cite{Rand2004, Rand2005, Rand2007}, \cite{HM2012} for further results, and to  \cite{CRX2023, HSZ2022, JOW2018,JRWZ2020,JWZ2023,JZZ2023, Ra2022, Ra2023}  for recent results on noncommutative martingales.

In the present paper, we mainly deal with noncommutative dual Doob inequalities. Let us first recall basic symbol and definitions. Throughout, let $(\mathcal{M},\tau)$ be a noncommutative probability space, that is, $\mathcal{M}$ is a semifinite von Neumann algebra equipped with a normal faithful semifinite trace satisfying $\tau(\mathbf{1})=1$, where $\mathbf{1}$ denotes the identity element in $\mathcal{M}$.  Let $(\mathcal{M}_k)_{k\geq1}$ be a increasing von Neumann subalgebras of $\mathcal{M}$. The conditional expectation from $\mathcal{M}$ onto $\mathcal{M}_k$ is denoted by $\mathcal{E}_k$, $k\geq1$. 
Recall that Junge's famous noncommutative Doob inequality \cite{Ju2002} was proved via the following dual Doob inequality (\cite[Theorem 0.1]{Ju2002}): Let $1\leq p<\infty$. Then for any sequence of positive operators $(x_k)_{k=1}^n$  in $L_p(\mathcal{M})$, there exists some  constant $c_p>0$ depends only on $p$ such that
\begin{equation}\label{dd-b1}
	\Big\| \sum_{k=1}^n\mathcal{E}_k(x_k)\Big\|_{L_p(\mathcal{M})}\leq c_p\Big\| \sum_{k=1}^nx_k\Big\|_{L_p(\mathcal{M})}.
\end{equation}
In \cite[Theorem 8(iii)]{JX2005}, the authors showed that optimal order  of this inequality is $O(p^2)$ as $p\to\infty$.  For $1\leq p\leq 2$, Junge and Xu \cite[Page 989]{JX2003} proved \eqref{dd-b1} with the constant $2$. Moreover, for $0<p\leq 1$, Junge and Xu \cite[Theorem 7.1]{JX2003} showed the following result: 
for any sequence of  positive   operators $(x_k)_{k=1}^n$ in $\mathcal{M}$, the following holds:
\begin{equation}\label{JX-l1}
	\Big\| \sum_{k=1}^nx_k\Big\|_{L_p(\mathcal{M})}\leq 2^{1/p} \Big\| \sum_{k=1}^n\mathcal{E}_k(x_k)\Big\|_{L_p(\mathcal{M})},\quad 0<p\leq 1.
\end{equation}
Very recently, the authors in \cite[Theorem 1.2]{JSWZ2023} studied the distributional dual Doob inequality, which is a generalization of \eqref{dd-b1}; Randrianantoanina \cite[Corollary 4.6]{Ra2022}
obtained \eqref{JX-l1} for noncommutative symmetric quasi-Banach space $E$, which is a interpolation space between $L_p$ and $L_q$ with $0<p<q<1$. 

In this paper, we continue this line of research, and prove the following two sharp noncommutative dual Doob inequalities. The proofs are provided in Section \ref{sec pf}. 
In the special case, when von Neumann algebra is a commutative von Neumann algebra, our results recover that of Wang \cite[Theorem 2]{Wa1991}. Our approach to the proof of Theorems \ref{best-c} and \ref{best-c-2} are completely different from that of \cite{Wa1991} and provides also alternative proof of the classical (commutative) sharp estimates. Our approach to the proofs of Theorems \ref{best-c} and \ref{best-c-2} is inspired by the approach of Junge and Xu \cite{JX2003}, however some technical improvements allow us obtaining the sharp estimates.

\begin{theorem}\label{best-c}
	Let $0<p\leq 1$. Then for any sequence of  positive   operators $(x_k)_{k=1}^n$ in $\mathcal{M}$, the following holds:
	\begin{equation}\label{ds}
		\Big\| \sum_{k=1}^nx_k\Big\|_{L_p(\mathcal{M})}\leq \frac{1}{p} \Big\| \sum_{k=1}^n\mathcal{E}_k(x_k)\Big\|_{L_p(\mathcal{M})}.
	\end{equation}
	The constant $1/p$ is best possible.
\end{theorem}

\begin{theorem}\label{best-c-2}
	Let $1\leq p\leq 2$. Then for any sequence of  positive   operators $(x_k)_{k=1}^n$ in $L_p(\mathcal{M})$, the following holds:
	\begin{equation}\label{ds-2}
		\Big\| \sum_{k=1}^n\mathcal{E}_k(x_k)\Big\|_{L_p(\mathcal{M})}\leq p\Big\| \sum_{k=1}^nx_k\Big\|_{L_p(\mathcal{M})}.
	\end{equation}
	The constant $p$ is best possible.
\end{theorem}

In the sequel, we consider applications of our sharp dual Doob inequalities.
For  $x\in L_1(\mathcal{M})$, the column  square  function and conditioned column square function are respectively  defined by 
$$S_{c,N}(x)^2=|\mathcal{E}_1(x)|^2+\sum_{k=2}^N|\mathcal{E}_{k}(x)-\mathcal{E}_{k-1}(x)|^2,$$
and 
$$s_{c,N}(x)^2=|\mathcal{E}_1(x)|^2+\sum_{k=2}^N\mathcal{E}_{k-1}(|\mathcal{E}_{k}(x)-\mathcal{E}_{k-1}(x)|^2).$$
The row square functions  $S_{r,N}(x)$ and $s_{r,N}(x)$ can be defined by taking adjoints.
A direct corollary of Theorem \ref{best-c} and Theorem  \ref{best-c-2}    is the following sharp noncommutative   martingale square function inequalities  which are the noncommutative generalization of \cite[Theorem 1]{Wa1991}. The inequality \eqref{les 2} was already shown in \cite[Theorem 4.11]{JRWZ2020}. 

\begin{corollary}
	Let $0< p\leq 4$. Then, for each $N\geq2$,
	\begin{equation}\label{les 2}
		\|S_{c,N}(x)\|_{L_p(\mathcal{M})}\leq \sqrt{\frac{2}{p}}\|s_{c,N}(x)\|_{L_p(\mathcal{M})}, \quad 0<p\leq 2,
	\end{equation}
	and
	$$\|s_{c,N}(x)\|_{L_{p}(\mathcal{M})}\leq \sqrt{\frac{p}{2}} \|S_{c,N}(x)\|_{L_p(\mathcal{M})},\quad 2\leq p\leq 4.$$
	The constants are all best possible. The same results hold for $S_{r,N}(x)$ and $s_{r,N}(x)$.
\end{corollary}

We can also get the following Burkholder-Gundy inequality with the help of the dual Doob inequality. Our proof here is different from its commutative version given in \cite[Theorem 3.2]{Bu1973}. The proof is provided in next section. 

\begin{corollary}\label{SI}
	Let $2< p\leq 4$. Then, for any $N\geq2$,
	$$\|S_{c,N}(x)\|_{L_p(\mathcal{M})}\leq c_p\|x\|_{L_p(\mathcal{M})},$$
	where $c_p=(2p\frac{p}{p-2})^{1/2}$. Here, $c_p\leq O(p^{1/2})$ as $p\geq3$, which is the same as the commutative setting.
\end{corollary}

We here recall a few well known properties of conditional expectations which are used in the proof below. Assume that $\mathcal{E}:\mathcal{M}\to\mathcal{N}$ is a conditional expectation, where $\mathcal{N}$ is a von Neumann subalgebra of $\mathcal{M}$. The following hold
\begin{enumerate}[{\rm (i)}]
	\item 	$\tau(\mathcal{E}(x))=\tau(x)$ for all $x\in L_1(\mathcal{M})$;
	\item $\mathcal{E}(y)=y$ for all $y\in L_1(\mathcal{N})$;
	\item if $x\in \mathcal{M}$ and $y\in \mathcal{N}$, then $\mathcal{E}(xy)=\mathcal{E}(x)y$ and $\mathcal{E}(yx)=y\mathcal{E}(x)$.
\end{enumerate}
These  properties also imply 
\begin{equation}\label{eq:self adjoint}
	\tau(\mathcal{E}(x)y))=\tau(\mathcal{E}(x)\mathcal{E}(y))=\tau(x\mathcal{E}(y)),\quad x,y \in \mathcal{M}. 
\end{equation}

The noncommutative Doob inequality was first proved by Junge \cite{Ju2002} for $1<p\leq\infty$; it was also shown in \cite[Theorem 8]{JX2005} that the best order for the inequality is $(p-1)^{-2}$ as $p\to1$ (for the commutative setting, the best order is $(p-1)^{-1}$ as $p\to1$). By Theorem \ref{best-c-2}, for $2\leq p\leq \infty$, we can establish  the following sharp noncommutative Doob inequality. The constant is just the same as in the commutative case. 

\begin{corollary}
	Let $2\leq p\leq \infty$. Then, for positive $x\in L_p(\mathcal{M})$,
	$$\|(\mathcal{E}_n(x))_{n\geq1}\|_{L_p(\mathcal{M},\ell_{\infty})}\leq p' \|x\|_{L_p(\mathcal{M})},$$
	where 
	$$\|(\mathcal{E}_n(x))_{n\geq1}\|_{L_p(\mathcal{M},\ell_{\infty})}=\inf\{\|a\|_{L_p(\mathcal{M})}: a\geq0, \mathcal{E}_n(x)\leq a,\; \forall n\geq1\}.$$
	The constant  $p'=p/(p-1)$,  which is the conjugate number of $p$,  is best possible. 
\end{corollary}

\begin{proof}
	Combining \cite[Remark 3.7]{Ju2002} and  \eqref{eq:self adjoint}, we write
	\begin{align*}
		\|(\mathcal{E}_n(x))_{n\geq1}\|_{L_p(\mathcal{M},\ell_{\infty})}&=\sup\Big\{\sum_{n\geq1}\tau(\mathcal{E}_n(x)y_n):y_n\geq0,\Big\|\sum_ny_n\Big\|_{L_{p'}(\mathcal{M})}\leq 1\Big\}\\
		&=\sup\Big\{\sum_{n\geq1}\tau(x\mathcal{E}(y_n)):y_n\geq0,\Big\|\sum_ny_n\Big\|_{L_{p'}(\mathcal{M})}\leq 1\Big\}.
	\end{align*}	
	Then, by    the H\"{o}lder inequality and Theorem  \ref{best-c-2}, we have
	\begin{align*}
		\|(\mathcal{E}_n(x))_{n\geq1}\|_{L_p(\mathcal{M},\ell_{\infty})}
		&\leq \sup\Big\{\|x\|_{L_p(\mathcal{M})}\Big\|\sum_n\mathcal{E}_n(y_n)\Big\|_{L_{p'}(\mathcal{M})}:y_n\geq0,\Big\|\sum_ny_n\Big\|_{L_{p'}(\mathcal{M})}\leq 1\Big\}\\
		&\leq p' \|x\|_{L_p(\mathcal{M})}.
	\end{align*}
	The constant is already sharp in the classical setting.
\end{proof}

The noncommutative Stein inequality was firstly proved by Pisier and Xu in \cite[Theorem 2.3]{PX1997}, and the best order  for it is $p$ as $p\to\infty$ (see e.g. \cite[Theorem 8]{JX2005}). However, the best order for commutative Stein inequality  is just $\sqrt{p}$ as $p\to\infty$ (see for instance \cite[Theorem 6]{JX2005}). With the help of Theorem \ref{best-c-2}, we obtain the following noncommutative Stein inequality with better constant for $4/3\leq p\leq 4$.
\begin{corollary}
	Let  $4/3\leq p\leq 4$. Then, for general elements $(x_k)_{k=1}^n$ in $L_p(\mathcal{M})$, we have
	$$\Big\|\Big(\sum_{k=1}^n|\mathcal{E}_k(x_k)|^2\Big)^{1/2} \Big\|_{L_p(\mathcal{M})}\leq \sqrt{\frac{p}{2\min\{p-1,1\}}} \Big\|\Big(\sum_{k=1}^n|x_k|^2\Big)^{1/2} \Big\|_{L_p(\mathcal{M})}.$$
\end{corollary}
\begin{proof}
	We first consider the case $2\leq p\leq 4$. According to Kadison's inequality (see e.g. \cite[Theorem 1]{Ka1952}), we know that for each $k\geq1$, $|\mathcal{E}_k(a)|^2\leq \mathcal{E}_k(|a|^2)$. 
	Then, by Theorem \ref{best-c-2}, we get
	\begin{align*}
		\mathrm{LHS}&\leq \Big\|\Big(\sum_{k=1}^n\mathcal{E}_k(|x_k|^2)\Big)^{1/2} \Big\|_{p}=\Big\|\sum_{k=1}^n\mathcal{E}_k(|x_k|^2) \Big\|_{p/2}^{1/2}\\
		&\leq \left(\frac{p}{2}\Big\|\sum_{k=1}^n|x_k|^2 \Big\|_{p/2}\right)^{1/2}
		=\sqrt{\frac{p}{2}}\Big\|\Big(\sum_{k=1}^n|x_k|^2 \Big)^{1/2}\Big\|_{p}.
	\end{align*}
	
	Now we turn the case $4/3\leq p\leq 2$. Actually, by duality, there exist $(y_k)_{k=1}^n\subset L_{p'}(\mathcal{M})$ with $\|(\sum_{k=1}^n|y_k|^2)^{1/2}\|_{p'}\leq 1$ such that (note that the conjugate number of $p$, $p'\in [2,4]$)
	\begin{align*}
		\mathrm{LHS}&=\sum_{k=1}^n\tau(\mathcal{E}_k(x_k^*)y_k)=\sum_{k=1}^n\tau(x_k^*\mathcal{E}_k(y_k))\\
		&\leq \Big\|\Big(\sum_{k=1}^n|x_k|^2\Big)^{1/2} \Big\|_p \Big\|\Big(\sum_{k=1}^n\mathcal{E}_k(|y_k|^2)\Big)^{1/2} \Big\|_{p'}\\
		&\leq  \sqrt{\frac{p'}{2}}\Big\|\Big(\sum_{k=1}^n|x_k|^2\Big)^{1/2} \Big\|_p =\sqrt{\frac{p}{2(p-1)}}\Big\|\Big(\sum_{k=1}^n|x_k|^2\Big)^{1/2} \Big\|_p.
	\end{align*}
	The proof is complete. 
\end{proof}

\section{Proof of the main results}\label{sec pf}
In this section, we provide the proofs of Theorem \ref{best-c} and Theorem \ref{best-c-2}. To this end, we begin with two basic lemmas.

\begin{lemma}\label{lem 1}
	Suppose that  positive operators $a,b\in \mathcal{M}$ have bounded inverses. If $0<p<1$, then 
	$$\|a\|_{L_p(\mathcal{M})}\leq \|b\|_{L_{p}(\mathcal{M})}^{1-p} \tau(ab^{p-1}).$$
\end{lemma}

\begin{proof}
	Take $q=\frac{2p}{1-p}$. Then it is clear that  
	$$\frac{1}{2p}=\frac{1}{2}+\frac{1}{q} .$$
	For each $n\geq1$,  by the H\"{o}lder inequality, we have
	\begin{equation*} 
		\begin{aligned}
			\|a\|_p&=\|a^{1/2}\|_{2p}^2\\
			& = \|a^{1/2}b^{\frac{1}{2}(p-1)}b^{\frac{1}{2}(1-p)}\|_{2p}^2\\
			&\leq \|a^{1/2}b^{\frac{1}{2}(p-1)}\|_2^2\|b^{\frac{1}{2}(1-p)}\|_{q}^2\\
			&=\tau(b^{\frac{1}{2}(p-1)}ab^{\frac{1}{2}(p-1)})\cdot\|b\|_{p}^{1-p}.
		\end{aligned}
	\end{equation*}
	The desired inequality now follows from the tracial property of $\tau$. 
\end{proof}

The next result is contained in the proof of \cite[Proposition 3.2]{BCPY2010}; see also the statement of preceding \cite[Lemma 2.3]{CRX2023}. In fact, in \cite[Lemma 7.3]{JX2003}, the authors showed the following result with a worse constant.   Here we provide a new and simpler  proof for it. 
\begin{lemma}\label{nc-blemma-2}
	Suppose that  positive operators $a,b\in \mathcal{M}$ have bounded inverses.	If $ a\leq b$  and $0<p \leq 1 $, 
	then 
	$$\tau\Big((b-a)b^{p-1}\Big)\leq \frac{1}{p} \tau(b^p-a^p).$$
\end{lemma}
\begin{proof}
	We only need to consider the case when $0<p<1$.  Recall that the classical Young inequality states that for positive numbers $A,B$ we have
	\begin{equation}\label{eq:Young}
		AB\leq \frac{1}{r}A^r+\frac{1}{r'}B^{r'},\quad 1<r<\infty,\frac{1}{r}+\frac{1}{r'}=1.
	\end{equation}
	By Lemma \ref{lem 1}, we write
	$$\tau(a^p)=\|a\|_p^p\leq \|b\|_p^{p(1-p)}\,[\tau(ab^{p-1})]^p.$$
	Then the above classical Young's inequality with $r=1/(1-p)$ gives us
	\begin{align*}
		\tau(a^p)&\leq (1-p)\|b\|_p^p+p\tau(ab^{p-1})\\
		&=(1-p)\tau(b^p) +p\tau(ab^{p-1})\\
		&=\tau(b^p)-p\tau\Big((b-a)b^{p-1}\Big).
	\end{align*}
	The claim follows immediately.
\end{proof}

We will also need the following elementary operator inequalities; see Proposition V.1.6 and Theorem V1.19 in \cite{Bh1997}.
\begin{lemma}\label{lem:elementary} Assume that $a,b\in \mathcal{M}$ are positive. 
	\begin{enumerate}[{\rm (i)}]
		\item If $a,b$ are invertible and $a\leq b$, then $b^{-1}\leq a^{-1}$.
		\item If $a\leq b$ and $0<r\leq 1$, then $a^r\leq b^r$.
	\end{enumerate}
\end{lemma}

Now we  prove Theorem \ref{best-c}. 
\begin{proof}[Proof of Theorem \ref{best-c}]
	The constant $1/p$ is already sharp in the classical  setting; see \cite[Theorem 2]{Wa1991}. Hence, we only need to show the desired inequality. Note that the desired inequality \eqref{ds} for $p=1$ is trivial, so we only  prove \eqref{ds} for $0<p<1$.
	For each $n\geq1$, set
	$$A_n=\sum_{1\leq k\leq n}x_k, \quad B_n=\sum_{1\leq k\leq n}\mathcal{E}_k(x_k), \quad B_0=0.$$ 
	Without loss of generality, we may assume that $A_n$ and $B_n$ are all invertible. 
	By Lemma \ref{lem 1},  we first write
	\begin{equation}\label{first-e}
		\|A_n\|_p\leq \|B_n\|_{p}^{1-p} \tau(A_n B_n^{p-1}).
	\end{equation}
	Since $0\leq B_{k-1}\leq B_k$  for each $k\geq1$, it follows from Lemma \ref{lem:elementary} that  
	$$B_{k}^{p-1}\leq  B_{k-1}^{p-1}.$$
	This inequality gives us
	\begin{equation}\label{eq:first estimate}
		\begin{aligned}
			\tau(A_n B_n^{p-1})&=\sum_{k=1}^n \tau(A_kB_{k}^{p-1})- \tau(A_{k-1}B_{k-1}^{p-1})\\
			&\leq \sum_{k=1}^n \tau(A_kB_{k}^{p-1})- \tau(A_{k-1}B_{k}^{p-1})\\
			&=\sum_{k=1}^n\tau ([A_k-A_{k-1}]B_k^{p-1})\\
			&\overset{\mathrm{a}}{=}\sum_{k=1}^n\tau ([B_k-B_{k-1}]B_k^{p-1}),
		\end{aligned}
	\end{equation}
	where in ``$\mathrm{a}$" we used \eqref{eq:self adjoint} and the fact that  $B_{k}$ is measurable with respect to $\mathcal{M}_k$ (hence $\mathcal{E}_k(B_k^{p-1})=B_k^{p-1}$). 
	Now, combining \eqref{eq:first estimate} and Lemma \ref{nc-blemma-2}, we have
	\begin{equation}\label{sec-e}
		\tau(A_n B_n^{p-1})\leq \frac{1}{p}\sum_{k=1}^n\tau(B_k^p-B_{k-1}^p)=\frac{1}{p}\|B_n\|_p^p.
	\end{equation}
	Combining \eqref{first-e} and \eqref{sec-e}, we finish the proof. 
\end{proof}

To prove  Theorem \ref{best-c-2}, we need the following lemma, which is motivated by our proof of Lemma \ref{nc-blemma-2}. In fact, for $1\leq p\leq 2$, Lemma \ref{lem big 1} below was proved in \cite[(7.6)]{JX2003} with a worse constant. 

\begin{lemma}\label{lem big 1}
	Assume that $0\leq a\leq b$. 	If  $1\leq p<\infty$, 
	then 
	$$\tau(b^p-a^p)\leq p\;\tau\Big((b-a)b^{p-1}\Big).$$
\end{lemma}
\begin{proof}
	We only need to prove the lemma when $1<p<\infty$. By the H\"{o}lder inequality and the Young inequality, we have the following estimate
	\begin{align*}
		\tau(ab^{p-1})&\leq \|ab^{p-1}\|_{L_1(\mathcal{M})}\\
		&\leq \|a\|_p \|b^{p-1}\|_{p'}\\
		&\overset{\eqref{eq:Young}}{\leq} \frac{1}{p}\|a\|_p^p+\frac{1}{p'}\|b^{p-1}\|_{p'}^{p'}\\
		&=\frac{1}{p}\tau(a^p)+(1-\frac{1}{p}) \tau(b^p)
	\end{align*}
	This means 
	$$\tau(ab^{p-1})\leq \frac{1}{p}\tau(a^p-b^p)+\tau(b^p),$$
	which is just the desired result. 
\end{proof}

Now we are ready to prove Theorem \ref{best-c-2}. 

\begin{proof}[Proof of Theorem \ref{best-c-2}]
	We here use the notations $A_n$ and $B_n$ introduced in the proof of Theorem \ref{best-c}.
	By Lemma \ref{lem big 1}, we obtain
	\begin{align*}
		\tau(B_n^p)&=\sum_{k=1}^n\tau(B_k^p)-\tau(B_{k-1}^p)\\
		&\leq p\sum_{k=1}^n\tau([B_k-B_{k-1}]B_{k}^{p-1})\\
		&\overset{\mathrm{b}}{=}p\sum_{k=1}^n\tau(x_kB_{k}^{p-1})\\
		&\overset{\mathrm{c}}{\leq}p\sum_{k=1}^n\tau(x_kB_{n}^{p-1})\\
		&=p\tau(A_nB_n^{p-1}).
	\end{align*}
	Here, ``b" is due to \eqref{eq:self adjoint} and  the fact that $\mathcal{E}_k(B_k)=B_k$ (actually $B_{k}\in L_p(\mathcal{M}_k)$); ``c" is due to Lemma \ref{lem:elementary} (ii). According to the above argument, using the H\"{o}lder inequality, we obtain
	\begin{align*}
		\|B_n\|_p^p=\tau(B_n^p)\leq p\|A_n\|_p\|B_n^{p-1}\|_{p'}=p\|A_n\|_p\|B_n\|_p^{p-1},
	\end{align*}
	which implies the desired inequality. The constant is already best possible; see \cite[Theorem 2]{Wa1991}.  The proof is complete. 
\end{proof}

Now we provide the proof of Corollary \ref{SI}.

\begin{proof}[Proof of Corollary \ref{SI}]
Recall that $2<p\leq 4$. Without loss of generality, let $x\in \mathcal{M}$. For simplicity, let us denote $$S_{c,N}^2=|\mathcal{E}_1(x)|^2+\sum_{k=2}^N|\mathcal{E}_{k}(x)-\mathcal{E}_{k-1}(x)|^2.$$
Note that $S_{c,k}^2\geq S_{c,k-1}^2$ for each $k$. Using Lemma \ref{lem big 1} with $b=S_{c,k}^2$ and $a=S_{c,k-1}^2$, we obtain
\begin{equation}\label{eq:S}
\begin{aligned}
	\tau(S_{c,N}^p)&=\sum_{k=1}^N\tau(S_{c,k}^p-S_{c,k-1}^p)\\
	&\overset{L. 10}{\leq} \frac{p}{2}\sum_{k=1}^N\tau([S_{c,k}^2-S_{c,k-1}^2]S_{c,k}^{p-2})\\
	&=\frac{p}{2}\sum_{k=1}^N\sum_{j=1}^{k}\tau([S_{c,k}^2-S_{c,k-1}^2][S_{c,j}^{p-2}-S_{c,j-1}^{p-2}])\\
	&=\frac{p}{2}\sum_{j=1}^{N}\tau([S_{c,N}^2-S_{c,j-1}^2][S_{c,j}^{p-2}-S_{c,j-1}^{p-2}])\\
	&=\frac{p}{2}\sum_{j=1}^{N}\tau(\mathcal{E}_j(S_{c,N}^2-S_{c,j-1}^2)[S_{c,j}^{p-2}-S_{c,j-1}^{p-2}]),
\end{aligned}
\end{equation}
where we used  \eqref{eq:self adjoint} and the fact  $S_{c,j}^{p-2}-S_{c,j-1}^{p-2}\in \mathcal{M}_j$.
It is easy to see that for each $1\leq j\leq N$, we have
\begin{equation}\label{eq:S2}
	\begin{aligned}
		\mathcal{E}_j(S_{c,N}^2-S_{c,j-1}^2)&=\mathcal{E}_j(|\mathcal{E}_N(x)-\mathcal{E}_{j-1}(x)|^2)\\
		&\leq 2\mathcal{E}_j(|\mathcal{E}_N(x)|^2+|\mathcal{E}_{j-1}(x)|^2)
		\leq 2 \mathcal{E}_j(|x|^2)+2\mathcal{E}_{j-1}(|x|^2).
	\end{aligned} 
\end{equation}
Also note that, by the fact $S_{c,j}^2\geq S_{c,j-1}^2$ and  Lemma \ref{lem:elementary} with $r=(p-2)/2$, for $2<p\leq 4$, each  $S_{c,j}^{p-2}-S_{c,j-1}^{p-2}$ is positive. Then, combining \eqref{eq:S} and \eqref{eq:S2}, we obtain
\begin{align*}
\tau(S_{c,N}^p)&\leq p\sum_{j=1}^N\tau(|x|^2[S_{c,j}^{p-2}-S_{c,j-1}^{p-2}])+p\sum_{j=1}^N\tau(\mathcal{E}_{j-1}(|x|^2)[S_{c,j}^{p-2}-S_{c,j-1}^{p-2}])\\
&=p\sum_{j=1}^N\tau(|x|^2S_{c,N}^{p-2})+p\tau(|x|^2\sum_{j=1}^N\mathcal{E}_{j-1}[S_{c,j}^{p-2}-S_{c,j-1}^{p-2}]).
\end{align*}
The H\"{o}lder inequality gives us that 
$$\tau(|x|^2S_{c,N}^{p-2})\leq \||x|^2\|_{\frac{p}{2}}\|S_{c,N}^{p-2}\|_{\frac{p}{p-2}}=\|x\|_p^2\|S_{c,N}\|_{p}^{p-2}.$$
Combining  H\"{o}lder's inequality and Theorem \ref{best-c-2},  we obtain 
\begin{align*}
	\tau(|x|^2\sum_{j=1}^N\mathcal{E}_{j-1}[S_{c,j}^{p-2}-S_{c,j-1}^{p-2}])&\leq \||x|^2\|_{\frac{p}{2}}\Big\|\sum_{j=1}^N \mathcal{E}_{j-1}[S_{c,j}^{p-2}-S_{c,j-1}^{p-2}]\Big\|_{\frac{p}{p-2}}\\
	&\leq \frac{p}{p-2}\|x\|_p^2\Big\|\sum_{j=1}^N S_{c,j}^{p-2}-S_{c,j-1}^{p-2}\Big\|_{\frac{p}{p-2}}=\frac{p}{p-2}\|x\|_p^2\|S_{c,N}\|_{p}^{p-2}.
\end{align*}
Thus we conclude from the above argument that 
$$\|S_{c,N}\|_p^p\leq 2p\frac{p}{p-2} \|x\|_p^2\|S_{c,N}\|_{p}^{p-2},$$
which gives us the desired inequality. The order now is the same as the commutative case (see \cite[Theorem 3.2]{Bu1973}).
\end{proof}

We conclude this paper with the following comments and questions.
\begin{remark}
Note that Lemma \ref{lem:elementary} (ii) is true for all $r\in (0,\infty)$ in the commutative setting.  Hence, in the commutative martingale setting, the proof of Theorem \ref{best-c-2} actually allows us to get the sharp  dual Doob inequality for the range $1\leq p<\infty$:
	$$\Big\|\sum_k\mathcal{E}_k(x_k)\Big\|_p\leq p \Big\|\sum_kx_k\Big\|_p. $$
	The fact that the constant is sharp was stated in \cite[Theorem 2]{Wa1991}.	However, in the noncommutative martingale setting,  we know from \cite{JX2005} that the inequality holds with the best order constant $cp^2$ as $p\to\infty$. We conjecture  that for all $2<p<\infty$,    Theorem \ref{best-c-2} holds with constant $p^2$.
\end{remark}

\begin{remark}
(i)	Combining \eqref{eq:S} and \eqref{eq:S2}, we can obtain 
	$$\|S_{c}(x)\|_{L_p(\mathcal{M})}\leq \sqrt{\frac{p}{2}}\|(\mathcal{E}_j(|x-\mathcal{E}_{j-1}(x)|^2))_{j\geq1}\|_{L_{p/2}(\mathcal{M},\ell_{\infty})}^{1/2},\quad 2< p\leq 4.$$
	The constant is the same as the commutative setting.  (ii)   The argument of Corollary \ref{SI} allows us to show, in the commutative setting,  for all $2<p<\infty$ (the reason is  same to the previous remark), 
	$$\|S(f)\|_p\leq c_p \|f\|_p, $$
	 where $c_p=O(p^{1/2})$ as $p\to\infty$ is the best order (see also \cite[Theorem 3.2]{Bu1973}). We also point out that the argument here is different with the one used in \cite[Theorem 3.2]{Bu1973}.
\end{remark}

\begin{remark}
	It was proved in \cite[Theorem 20.1]{Bu1973} that for any concave function $\Phi$ and positive functions $(f_k)_k$, we have 
	$$\mathbb{E}(\Phi[\sum_{k}f_k])\leq 2 \mathbb{E}(\Phi[\sum_k\mathcal{E}_k(f_k)]).$$
	This inequality also holds for each $p$-convex and $q$-concave Orlicz function $\Phi$ with $0<p<q<1$; see \cite[Corollary 4.6]{Ra2022}. It is natural to ask whether a full noncommutative analogue of \cite[Theorem 20.1]{Bu1973} holds.
\end{remark}

%


\providecommand{\bysame}{\leavevmode\hbox to3em{\hrulefill}\thinspace}
\providecommand{\MR}{\relax\ifhmode\unskip\space\fi MR }
\providecommand{\MRhref}[2]{%
	\href{http://www.ams.org/mathscinet-getitem?mr=#1}{#2}
}
\providecommand{\href}[2]{#2}

\end{document}